\documentclass[12pt,letterpaper,reqno]{amsart}
\usepackage{tikz}
\usetikzlibrary{positioning, shapes.geometric, arrows.meta}
\usepackage{amssymb}
\usepackage{amsmath}
\usepackage{amsthm}
\usepackage{amsfonts}
\IfFileExists{bbm.sty}{\usepackage{bbm}}{}
\usepackage{enumitem}
\usepackage{pgfplots}
\pgfplotsset{compat=1.18}
\usepackage{booktabs,tabularx,array}
\usepackage{graphicx}
\usepackage[T1]{fontenc}
\usepackage{doi}
\usepackage{float}

\usepackage[dvipsnames]{xcolor}
\usepackage{bookmark}
\usepackage{hyperref}
\hypersetup{colorlinks=true,
linkcolor=RoyalBlue,
citecolor=ForestGreen!65!white,
urlcolor=BrickRed}
\allowdisplaybreaks

\newtheorem{thm}{Theorem}[section]
\newtheorem{lem}[thm]{Lemma}
\newtheorem{prop}[thm]{Proposition}

\newtheorem{conj}[thm]{Conjecture}
\theoremstyle{definition}

\numberwithin{equation}{section}

\newcommand{\C}{\mathbb C}
\newcommand{\R}{\mathbb R}
\newcommand{\Z}{\mathbb Z}

\newcommand{\dd}{\,d}
\newcommand{\dA}{\,dA}
\newcommand{\db}{\bar\partial}
\newcommand{\p}{\partial}
\newcommand{\calR}{\mathcal R}
\newcommand{\calC}{\mathcal C}
\newcommand{\calU}{\mathcal U}
\DeclareMathOperator{\dist}{dist}
\DeclareMathOperator{\supp}{supp}
\DeclareMathOperator{\interior}{int}
\DeclareMathOperator{\Rea}{Re}
\newcommand{\norm}[1]{\left\lVert#1\right\rVert}
\setlist[enumerate]{label=\textup{(\roman*)},
leftmargin=*,itemsep=3pt,topsep=4pt}
\makeatother

\begin{document}

\title[A counterexample to Fuchs's conjecture]
{A counterexample to Fuchs's conjecture}

\author[A.~Eremenko]{Alexandre Eremenko}

\address{Mathematics Department, Purdue University,
	West Lafayette, IN 47907, USA}
\email{eremenko@purdue.edu}

\author[T.~Zhang]{Teng Zhang}
\address{School of Mathematics and Statistics, Xi'an Jiaotong University,
Xi'an 710049, P. R. China}
\email{teng.zhang@stu.xjtu.edu.cn}

\begin{abstract}
For every $\rho\in(0,1/2)$, we construct an entire function $F$
of order and lower order $\rho$ whose logarithmic derivative
has zero as a deficient value, that is,
$\delta(0,F'/F)>0$.
This disproves an old conjecture of W.~H.~J.~Fuchs.
\end{abstract}

\maketitle

\tableofcontents

\section{Introduction}\label{sec:intro}
We use the standard notation
of Nevanlinna theory; see, for example, \cite{Hay64}.
For a meromorphic function $g$, we write $\rho(g)$ and $\lambda(g)$
for its order and lower order, respectively. For an entire function $g$,
$M(r,g)=\max_{|z|=r}|g(z)|$.

In \cite[Problem~22]{Ehr68}
Wolfgang Fuchs conjectured that for transcendental meromorphic functions
$f$ of order less than $1$, we have $\delta(0,f'/f)=0$.
Goldberg and Korenkov \cite{GK80} constructed counterexamples
among meromorphic functions of every prescribed order
\(0\leq\rho<1\), and among entire functions of every prescribed order
\(1/2<\rho<1\). So we consider the following formulation of Fuchs's
conjecture in the range \(\rho<1/2\).

\begin{conj}[Fuchs]\label{conj:Fuchs}
Let $F$ be a transcendental entire function satisfying $\rho(F)<1/2$. Then $
  \delta\!\left(0,{F'}/{F}\right)=0.$
\end{conj}

Goldberg and Korenkov proved Conjecture~\ref{conj:Fuchs}
for functions of completely regular growth in the
sense of Levin and Pfluger, and for functions of zero order satisfying
some
regularity conditions.

Eremenko, Langley, and Rossi
\cite[Corollary~1.10]{ELR94}
proved that, for an entire function \(F\) of order \(\rho(F)<1/2\)
and lower order \(\lambda(F)\),
\begin{equation}\label{elr}
\delta(0,F'/F)\leq 1-\cos\pi\lambda(F),
\end{equation}
which in particular proves Conjecture~\ref{conj:Fuchs} when
\(\lambda(F)=0\).

Miles and Rossi \cite{MR01} obtained a further upper bound for
$\delta(0,F'/F)$ when the order $\rho$ is sufficiently small and
positive. Their estimate improves \eqref{elr} 
(with $\rho$ instead of $\lambda$) for sufficiently small positive $\rho$.
Langley and Rossi \cite[Corollary~1.1]{LR04} proved the
conjecture for the class of transcendental entire functions
of order at most one, convergence class, whose zeros accumulate
to a ray.

Notice that, for a nonconstant entire function \(F\) of order less than \(1\)
with \(F(0)\neq0\),
\[
\mathcal L(z):=\frac{F'(z)}{F(z)}
=\sum_j\frac{m_j}{z-\zeta_j},
\]
where \(\zeta_j\) are the distinct zeros of \(F\) and \(m_j\) are
their multiplicities. Since \(F\) has order less than \(1\),
\[
\sum_j\frac{m_j}{|\zeta_j|}<\infty,
\]
and the series converges locally uniformly away from its poles.
Up to complex conjugation and a nonzero real factor, this series
represents the gravitational field generated by point masses $m_j$
at $\zeta_j$ in two dimensions.\footnote{An analogous
three-dimensional model consists of parallel, uniformly charged
wires perpendicular to the plane, with charge densities proportional
to $m_j$.}
Lee Rubel \cite[Problem~7.78]{HL19} asked whether a series of this
form, with infinitely many poles and all residues
equal to $1$, must have zeros (equilibrium points of the force).
For positive integer residues, including the case of unit residues,
Clunie, Eremenko, and Rossi \cite[Theorem~2.1]{CER93}
proved that such a function has infinitely many zeros.
Related questions with more general masses and in $\R^n$ are studied
in \cite{CER93,ELR94,LR04}.

The property $\delta(0,\mathcal L)>0$ is naturally expressed in terms
of integrated counting functions.
Let $N(r,\mathcal L)$ and $N(r,0,\mathcal L)$ denote the Nevanlinna counting
functions of the poles and zeros of $\mathcal L$, respectively.
Since $\sum_j m_j/|\zeta_j|<\infty$, the theorem of M.~V.~Keldysh
stated in \cite[Ch. V, Thm. 6.1]{GO08}
gives
\[
m(r,\mathcal L)=o(1),\qquad T(r,\mathcal L)=N(r,\mathcal L)+o(1).
\]
Consequently,
\[
\delta(0,\mathcal L)
=
1-\limsup_{r\to\infty}
\frac{N(r,0,\mathcal L)}{N(r,\mathcal L)}.
\]
Thus $\delta(0,\mathcal L)>0$ means that there exists $\eta>0$ such that
\[
N(r,0,\mathcal L)\leq(1-\eta)N(r,\mathcal L)
\]
for all sufficiently large $r$.

As mentioned in \cite{Ere09}, the construction of a counterexample
to Conjecture~\ref{conj:Fuchs} essentially boils down to constructing
a nonconstant subharmonic function of order $\rho<1/2$ which is
locally constant on an open set intersecting every circle $|z|=r$
for all sufficiently large $r$.

In this paper, we establish the following theorem,
which disproves Conjecture~\ref{conj:Fuchs}.

\begin{thm}\label{thm:main}
For each $0<\rho<1/2$, there exist a transcendental entire
function $F$ and constants $a,A,B>0$ such that,
for every sufficiently large $r$,
\begin{equation}\label{eq:main-growth}
 a r^\rho\le \log M(r,F)\le A r^\rho,
 \qquad
 m\!\left(r,\frac{F}{F'}\right)\ge B r^\rho.
\end{equation}
Consequently, $F$ has order and lower order $\rho$, and
$\delta\!\left(0, {F'}/F\right)>0.$
\end{thm}

\noindent\textbf{Sketch of the construction.}
In Sections~\ref{sec:prelim}--
\ref{sec:measures} we construct a perforated compact set $K$ and a positive
measure $\sigma$ with the point-evaluation identity needed for the
self-similar construction. In Section~\ref{sec:potential},
Proposition~\ref{prop:power-potential} we produce a nonconstant subharmonic
function $U$ such that
\[
 U(0)=0,
 \qquad
 U(2z)=2^\rho U(z),
\]
and $U$ is locally constant on a union of disjoint disks invariant under the  map $z\mapsto 2z$.
Lemma~\ref{lem:uniform-geometry} shows that a slightly smaller family of
these disks is an open set meeting every circle centered at the origin.
Together with the boundary-density estimate
\eqref{eq:circle-density}, these properties are precisely what we need in the final step.
In Section~\ref{sec:from-subharmonic} we convert this subharmonic
function into the required entire function and this completes the proof of
Theorem~\ref{thm:main}.

\medskip
\noindent\textbf{Organization of the paper.}
Section~\ref{sec:prelim} contains the conventions and two elementary
lemmas. Section~\ref{sec:geometry} constructs the
removable radial model and its disk approximations.
Section~\ref{sec:measures} constructs the positive boundary measure.
Section~\ref{sec:potential} constructs the self-similar logarithmic
potential and proves the all-radius geometric property. Finally,
Section~\ref{sec:from-subharmonic} completes the passage from the
subharmonic function to an entire counterexample.

\medskip
\noindent\textbf{Acknowledgments and AI tools disclosure.} We thank Mikhail Sodin for helpful
comments.

Teng Zhang is supported by the China Scholarship Council, the Young Elite
Scientists Sponsorship Program for PhD Students (China Association for Science
and Technology), and the Fundamental Research Funds for the Central
Universities at Xi'an Jiaotong University (Grant No.~xzy022024045).

ChatGPT was used for English-language editing, proofreading, and grammatical
correction, and as an exploratory tool for discussing possible approaches to
selected results under the authors’ mathematical supervision and guidance. The
authors take full responsibility for all mathematical arguments and for the
accuracy and correctness of the final manuscript.

\section{Conventions and elementary lemmas}\label{sec:prelim}

We write
\[
 dA=dx\,dy,\qquad
 \Delta=\partial_x^2+\partial_y^2,\qquad
 \p=\tfrac12(\partial_x-i\partial_y),\qquad
 \db=\tfrac12(\partial_x+i\partial_y),
\]
so that
\[
 \p\db=\tfrac14\Delta.
\]
For a subharmonic function $v$, its Riesz measure is
$(2\pi)^{-1}\Delta v$. Subharmonic functions are always understood in
their upper-semicontinuous representatives.
We write $ds$ for Euclidean arclength on the specified curve and
$\delta_a$ for the unit point mass at $a$. For a signed measure
$\lambda$, $|\lambda|$ denotes its total variation. If $P$ is a measurable map
and $\lambda$ is a measure on its domain, then $P_*\lambda$ denotes
its push-forward, defined by
$(P_*\lambda)(A)=\lambda(P^{-1}(A))$ for measurable sets $A$.

For a compact set $K\subset\C$, let $\calR(K)$ denote the set of rational
functions whose finite poles lie outside $K$. We shall work with
$\calR(K)$ itself, rather than with its uniform closure. We write
$\norm{h}_K=\sup_{z\in K}|h(z)|$ for the uniform norm on $K$.
For $a\in\C$ and $s>0$, let
\[
 B(a,s)=\{z\in\C:|z-a|<s\}.
\]
If $D=B(a,s)$ and $0<t<1$, we write
\[
 D(t)=B(a,ts)
\]
for the concentric subdisk. Scalar multiplication of sets always
refers to dilation about the origin.

Unless explicitly fixed, $C>0$ denotes a constant that may change
from one estimate to the next. The constants implicit in $O(\cdot)$
and $\lesssim$ may depend on the fixed order, the finite geometric
model, and auxiliary parameters already fixed, but are independent
of the dilation index. When a measure is initially defined on
$\C\setminus\{0\}$, we regard it as a Radon measure on $\C$ only after
local finiteness at the origin has been verified.


We begin with the elementary facts from Nevanlinna theory that will be used in
the final step of the proof. They also clarify the lower bound in
\eqref{eq:main-growth}.

\begin{lem}\label{lem:characteristic}
For nonzero meromorphic functions $g$ and $h$, and for $r\ge1$,
\begin{equation}\label{eq:characteristic-rules}
 T(r,1/g)=T(r,g)+O(1),\qquad
 T(r,gh)\le T(r,g)+T(r,h).
\end{equation}
If $F$ is a nonconstant entire function, then, for large $r$,
\begin{equation}\label{eq:logder-characteristic}
 T\!\left(r,\frac{F'}F\right)
 \le 2\log M(2r,F)+O(\log r).
\end{equation}
\end{lem}

\begin{proof}
Recall that
\[
 m(r,g)=\frac1{2\pi}\int_0^{2\pi}
        \log^+|g(re^{i\theta})|\,\dd\theta,
 \qquad
 T(r,g)=m(r,g)+N(r,g),
\]
where $N(r,g)$ is the integrated counting function of the poles of
$g$, with the usual contribution from a pole at the origin.
Jensen's formula (see, for example, \cite[\S1.1]{Hay64}) gives
$$
 m(r,g)-m(r,1/g)
   =N(r,1/g)-N(r,g)+O(1).
$$
Indeed, after factoring out the zeros and poles of $g$ in the disk,
the remaining factor is zero-free, and the mean-value property applied
to the logarithm of its modulus gives the identity above. A zero or
pole at the origin contributes the corresponding multiple of
$\log r$. Rearranging yields
\[
 T(r,1/g)=T(r,g)+O(1).
\]
The product estimate follows immediately from
\[
 \log^+|gh|\le \log^+|g|+\log^+|h|
\]
together with the corresponding inequality for the pole-counting
functions. This proves \eqref{eq:characteristic-rules}.

Now let $F$ be entire. For $|z|=r$, Cauchy's estimate \cite[p.~73, Cauchy's Estimate~2.14]{Con78} on
$B(z,r)\subset B(0,2r)$ gives
\[
 |F'(z)|\le \frac{M(2r,F)}{r},
\]
and hence
\[
 T(r,F')\le \log M(2r,F)+O(\log r).
\]
Applying \eqref{eq:characteristic-rules} to $
 {F'}/F=F'\cdot1/F$
and using
\[
 T(r,1/F)=T(r,F)+O(1),
 \qquad
 T(r,F)\le \log M(r,F),
\]
we obtain
\[
 T\!\left(r,\frac{F'}F\right)
 \le 2\log M(2r,F)+O(\log r),
\]
which is \eqref{eq:logder-characteristic}.
\end{proof}

The other elementary reduction concerns positive measures. We state it
explicitly to distinguish positivity from a representation by a signed
measure.

\begin{lem}\label{lem:positive-extension}
Let $X$ be a compact Hausdorff space, and let $\mathcal A$ be a
real linear subspace of $C(X,\R)$ containing the constants. If a linear functional $\Lambda$ on $\mathcal A$
is nonnegative on the nonnegative functions in $\mathcal A$, then there is a
positive finite measure $\lambda$ on $X$ such that $\Lambda(f)=\int f\dd\lambda$ for
every $f\in\mathcal A$ and $\lambda(X)=\Lambda(1)$.
\end{lem}

\begin{proof}
We use the standard extension argument; see
\cite[p.~309, Lemma~2.1]{BD59}.
We equip $\mathcal A$ with the uniform norm inherited from
$C(X,\R)$, and write $\|\Lambda\|$ for the corresponding operator
norm. If $f\in\mathcal A$, then
\[
 -\norm f_\infty\,1\le f\le \norm f_\infty\,1,
\]
and positivity gives
\[
 |\Lambda(f)|\le \Lambda(1)\norm f_\infty.
\]
Hence $\Lambda$ is bounded and $\norm \Lambda=\Lambda(1)$. By the Hahn--Banach theorem,
$\Lambda$ extends to a linear functional $\widetilde \Lambda$ on $C(X,\R)$ with $
 \norm{\widetilde \Lambda}=\norm \Lambda=\Lambda(1).$
By the Riesz representation theorem, there is a finite signed measure
$\lambda$ on $X$ such that
\[
 \widetilde \Lambda(f)=\int_X f\,\dd\lambda,
 \qquad f\in C(X,\R).
\]
Since
$
 \lambda(X)=\widetilde \Lambda(1)=\Lambda(1)
          =\norm{\widetilde \Lambda}=|\lambda|(X),
$
the negative part of $\lambda$ vanishes. Thus $\lambda$ is positive,
and $\lambda(X)=\Lambda(1)$.
\end{proof}

\section{A removable set with large radial projection and its 
approximants}\label{sec:geometry}

The geometric model is required to have a radial projection longer than
a dyadic fundamental interval, while its distinct dyadic dilates remain
pairwise disjoint. A suitable modification of Hallstrom's four-corner
construction \cite[pp.~455--456]{Hal74} provides both properties.

We use the following notation throughout this section and retain it
in the subsequent construction. Let
\[
 \calC=\left\{\sum_{n\ge1}3\varepsilon_n4^{-n}:
                       \varepsilon_n\in\{0,1\}\right\},
 \qquad
 Q=\calC+i\calC=\{x+iy:x,y\in\calC\}.
\]
Thus $Q$ denotes the four-corner Cantor set, not the unit square.
Its finite-stage approximants are
\[
 \calC_0=[0,1],\qquad
 \calC_N=\tfrac14\calC_{N-1}\cup
             \bigl(\tfrac34+\tfrac14\calC_{N-1}\bigr),\qquad
 Q_N=\calC_N+i\calC_N\quad(N\ge1).
\]
We also put $Q_0=[0,1]+i[0,1]$. Then
$Q=\bigcap_{N\ge0}Q_N$, and $Q_N$ consists of $L_N=4^N$ closed squares
of side length $h_N=4^{-N}$, denoted by
$Q_{N,1},\ldots,Q_{N,L_N}$. For $\ell>0$, define
\[
 T_\ell(z)=\exp\left(\frac{\ell}{3}(1-2i)z+\frac{i\ell}{6}\right),
 \qquad E_\ell=T_\ell(Q).
\]

\begin{lem}\label{lem:skeleton}
There exist parameters $\ell>0$ and $\theta_0\in(0,\pi/2)$ such that
the compact set $E=E_\ell=T_\ell(Q)$ lies in the sector
$|\arg w|<\theta_0$ and satisfies
\begin{equation}\label{eq:radial-model}
 \{|w|:w\in E\}=[1,e^\ell],\qquad \log2<\ell<2\log2.
\end{equation}
The sets $2^jE$, $j\in\Z$, are pairwise disjoint, and each is locally
removable for bounded holomorphic functions. Every bounded holomorphic
function on $\C\setminus Z$ is constant, where
\begin{equation}\label{eq:closed-skeleton}
 Z=\{0\}\cup\bigcup_{j\in\Z}2^jE.
\end{equation}
\end{lem}

\begin{proof}
The four-corner set $Q$ has zero analytic capacity; this is the
classical example of Garnett
\cite[Sect.~1, pp.~696--698]{Gar70}.  Hallstrom
\cite[pp.~455--456]{Hal74} used this set, followed by a suitable
affine change of variables and an exponential map, to construct a
compact removable set having full radial projection over a prescribed
annular interval.  We use the same basic mechanism, with the additional
choice of parameters needed below to separate all dyadic copies.

We first determine the radial projection of $E_\ell$. Since
$
 \calC+2\calC=[0,3],
$
the real part of
\[
 \frac{\ell}{3}(1-2i)z+\frac{i\ell}{6},
 \qquad z\in Q,
\]
ranges over $[0,\ell]$. Hence
\[
 \{|w|:w\in E_\ell\}=[1,e^\ell].
\]
Moreover, the imaginary part of the exponent ranges over
$[-\ell/2,\ell/2]$. If $\ell$ is sufficiently close to $\log 2$,
then $T_\ell$ is univalent on a neighborhood of the closed unit square $Q_0$.
Consequently, local removability of $Q$ is preserved under $T_\ell$:
a bounded holomorphic function defined off $E_\ell$ pulls back to one
defined off $Q$, extends across $Q$, and then pushes forward to an
extension across $E_\ell$.

We next arrange separation of the dyadic copies. At $\ell=\log2$,
the radial intervals of $E_\ell$ and $2E_\ell$ meet only at radius $2$.
The corresponding points are
\[
 2e^{-i\ell/6}
 \qquad\text{and}\qquad
 2e^{i\ell/6},
\]
respectively, and are distinct. Thus $E_\ell$ and $2E_\ell$ are
disjoint at $\ell=\log2$. Since $T_\ell$ depends uniformly on $\ell$
on $Q$, the same remains true for all $\ell>\log2$ sufficiently close
to $\log2$. Fix such an $\ell$ with $\ell<2\log2$, and set
$E=E_\ell$. Since $e^\ell<4$, dyadic copies whose indices differ by
at least two are separated by their radial ranges. It follows that
the sets $2^jE$, $j\in\Z$, are pairwise disjoint.

Finally, consider the set $Z$ defined in \eqref{eq:closed-skeleton}.
The set $Z$ is closed, and every compact annulus avoiding the origin
meets only finitely many of the sets $2^jE$. Hence a bounded
holomorphic function on $\C\setminus Z$ extends successively across
all these locally removable pieces. The resulting function is
holomorphic on $\C\setminus\{0\}$ and remains bounded near the origin,
so the singularity at $0$ is removable. The extension is therefore a
bounded entire function and hence is constant by Liouville's theorem.
\end{proof}

For the remainder of the paper, fix the value of $\ell$ chosen in the
proof of Lemma~\ref{lem:skeleton}, and retain $E=E_\ell=T_\ell(Q)$
and the set $Z$ in \eqref{eq:closed-skeleton}.
We next approximate $E$ by a finite family of disks. The
approximation is fixed before the order-dependent measure is introduced;
although the number of disks may be large, it remains finite.

\begin{lem}\label{lem:disk-model}
For every sufficiently large integer $N$, there are disks
$H_{N,1},\ldots,H_{N,L_N}$, indexed by the squares
$Q_{N,1},\ldots,Q_{N,L_N}$ above, with the following properties. Their union
contains $E$ and converges to $E$ in Hausdorff distance. The closed disks
\begin{equation}\label{eq:all-hole-disks}
 2^j\overline{H_{N,k}},\qquad j\in\Z,\quad 1\le k\le L_N,
\end{equation}
are pairwise disjoint. They lie in a fixed open sector in the right half-plane, and
their radii are less than $1/100$ of the moduli of their centers. For
each fixed $N$, a common enlargement factor greater than one preserves
disjointness of the entire family \eqref{eq:all-hole-disks}.
\end{lem}

\begin{proof}
Recall that $Q_N$ is the $N$th-stage approximation to the four-corner
set $Q=\calC+i\calC$, and that each square $Q_{N,k}$ has side length
$h_N=4^{-N}$. Let $\xi_{N,k}$ be the center of $Q_{N,k}$ and define
\[
 H_{N,k}=B\left(T_\ell(\xi_{N,k}),
                  \frac34|T_\ell'(\xi_{N,k})|h_N\right),
 \qquad 1\leq k\leq L_N.
\]
Since $T_\ell$ is holomorphic on a neighborhood of $Q_0$,
Taylor's formula gives, uniformly in $k$,
\[
 \max_{z\in Q_{N,k}}
 |T_\ell(z)-T_\ell(\xi_{N,k})|
 \le \frac{\sqrt2}{2}|T_\ell'(\xi_{N,k})|h_N+O(h_N^2).
\]
Because $\sqrt2/2<3/4$, the image of each square $Q_{N,k}$ is
contained in $H_{N,k}$ for all sufficiently large $N$. It follows at once that the
union of these disks contains $E$ and converges to $E$ in Hausdorff
distance.

We next verify that the disks have pairwise disjoint closures.
Since $T_\ell$ is univalent on a neighborhood of $Q_0$
and $T_\ell'$ does not vanish there, the divided difference
\[
 \Psi(z,z')=
 \begin{cases}
 \dfrac{T_\ell(z)-T_\ell(z')}{z-z'},& z\ne z',\\[6pt]
 T_\ell'(z),& z=z',
 \end{cases}
\]
is continuous and nonzero on the compact set \(Q_0\times Q_0\).
Hence compactness gives $m_*>0$ such that
\[
 |T_\ell(z)-T_\ell(z')|\ge m_*|z-z'|
\]
for all $z,z'\in Q_0$. Let $\xi$ and $\xi'$ denote the centers of
two distinct squares among the $Q_{N,k}$. They are separated by at
least $3h_N$.
If $|\xi-\xi'|\le h_N^{1/2}$, Taylor's formula gives, uniformly over
these pairs,
\[
 \begin{split}
 |T_\ell(\xi)-T_\ell(\xi')|
   &=(1+O(h_N^{1/2}))|T_\ell'(\xi)|\,|\xi-\xi'|,\\
 |T_\ell'(\xi')|&=(1+O(h_N^{1/2}))|T_\ell'(\xi)|.
 \end{split}
\]
Consequently, the distance between the image centers is at least
$(3+O(h_N^{1/2}))|T_\ell'(\xi)|h_N$, whereas the sum of the radii is
$(3/2+O(h_N^{1/2}))|T_\ell'(\xi)|h_N$.
If $|\xi-\xi'|>h_N^{1/2}$, the distance between the image centers
is at least $m_*h_N^{1/2}$, whereas the sum of the radii is $O(h_N)$.
Thus all the basic disks have pairwise disjoint closures for
sufficiently large $N$.

It remains to compare different dyadic scales. By scaling, it is enough
to compare a basic disk with the copies at relative scales $2^m$,
$m\in\Z$. If $|m|\ge2$, the corresponding copies are separated by their
radial ranges for all sufficiently large $N$. The two adjacent cases
$m=\pm1$ follow from the positive separation between $E$ and $2E$
established in Lemma~\ref{lem:skeleton}. Hence the entire family
\eqref{eq:all-hole-disks} has pairwise disjoint closures.

The sector condition and the bound on the ratio of each radius to the
modulus of its center follow from the corresponding properties of
$E$ once $N$ is sufficiently large. Finally, after normalizing one
disk to the basic scale, only finitely many relative dyadic scales can
come near it. The minimum separation over these finitely many
configurations is therefore positive, and a common enlargement factor
greater than one may be chosen while preserving pairwise disjointness.
\end{proof}

The outer boundary plays no direct role, but it must be chosen
compatibly with all dyadic holes, since its dilates will later contribute
to the support of the measure.

\begin{lem}\label{lem:outer-boundary}
After choosing $R_N\to\infty$ sufficiently large, there is a smooth Jordan
domain $\Omega_N$ satisfying
\begin{equation}\label{eq:outer-sandwich}
 B(0,0.9R_N)\subset\Omega_N\subset B(0,1.1R_N)
\end{equation}
whose boundary avoids every closed disk in \eqref{eq:all-hole-disks}. Put
\begin{equation}\label{eq:perforated-compact}
 K_N=\overline{\Omega_N}\setminus
             \bigcup_{\substack{j\in\Z\\1\leq k\leq L_N}}2^jH_{N,k},\qquad
 S(w)=w/2,
 \qquad p=-1,\quad q=-1/2.
\end{equation}
Then $K_N$ is compact, $S(K_N)\subset K_N$, and $p,q\in\interior K_N$. The
domain $\interior K_N$ is connected. Every compact subset of $\C\setminus Z$
has a neighborhood contained in $\interior K_N$ for all sufficiently large
$N$.
\end{lem}

\begin{proof}
Choose $R_N$ sufficiently large so that
\[
 \bigcup_{k=1}^{L_N}\overline{H_{N,k}}
 \subset B(0,0.9R_N).
\]
Slightly enlarge the disks from Lemma~\ref{lem:disk-model}, preserving
the pairwise disjointness of their closures. After a small perturbation
of $R_N$, we may assume that the circle $|w|=R_N$ meets the boundaries
of all enlarged disks transversely. Only finitely many such disks meet
this circle.

Let $G_N$ be the union of $B(0,R_N)$ with all enlarged disks that meet
$|w|=R_N$. Each of these disks intersects $B(0,R_N)$ in a nonempty
connected lens, and the disks are mutually disjoint. Consequently, the
outer boundary of $G_N$ is a Jordan curve consisting of finitely many
circular arcs, and it is disjoint from all the original closed holes.
Smoothing its finitely many corner points in sufficiently small
neighborhoods gives a smooth Jordan domain $\Omega_N$ whose boundary
still avoids every closed disk in \eqref{eq:all-hole-disks}. Since the
radius of each hole is less than one hundredth of the modulus of its
center, the above modifications may be confined to the annulus
\[
 B(0,1.1R_N)\setminus \overline{B(0,0.9R_N)}.
\]
Thus \eqref{eq:outer-sandwich} holds.

The set removed from $\overline{\Omega_N}$ in
\eqref{eq:perforated-compact} is open, and hence $K_N$ is compact.
If $w\in K_N$, then \eqref{eq:outer-sandwich} gives
\[
 |S(w)|=\frac{|w|}{2}\le 0.55R_N<0.9R_N,
\]
so $S(w)\in\Omega_N$. Since the family of holes is invariant under
dyadic scaling, $w$ lying outside every hole implies that $S(w)$ also
lies outside every hole. Hence
\[
 S(K_N)\subset K_N.
\]
The points $p=-1$ and $q=-1/2$ lie in the left half-plane, away from
all the holes, and for large $N$ they belong to the inner disk
$B(0,0.9R_N)$. Thus $p,q\in\interior K_N$.

We next prove that $\interior K_N$ is connected. The closed holes
contained in $\Omega_N$ can accumulate only at the origin. Given
$x,y\in\interior K_N$, choose a compact path joining $x$ to $y$ in
$\Omega_N\setminus\{0\}$. A sufficiently small neighborhood of this
path meets only finitely many holes. Replacing each portion of the path
that crosses one of these holes by a short detour through its exterior
collar produces a path from $x$ to $y$ contained entirely in
$\interior K_N$. Thus $\interior K_N$ is connected.

Finally, let $\mathcal K$ be a compact subset of $\C\setminus Z$. Then $\mathcal K$ is
contained in a compact annulus avoiding the origin and satisfies
$
 \dist(\mathcal K,Z)>0.
$
Only finitely many dyadic scales can meet a slightly larger annulus.
At each of these scales, the corresponding disk families converge in
Hausdorff distance to the appropriate dilate of $E$, whereas the outer
boundary of $\Omega_N$ tends to infinity. It follows that, for all
sufficiently large $N$, a fixed neighborhood of $\mathcal K$ is disjoint from
the holes and contained in $\Omega_N$. Hence that neighborhood is
contained in $\interior K_N$, proving the final assertion.
\end{proof}

\section{From close evaluations to positive boundary
measures}\label{sec:measures}

The role of removability is captured by the following limiting property:
bounded rational functions on the perforated sets can no longer distinguish
two fixed points in the left half-plane.

\begin{lem}\label{lem:close-evaluations}
For the compacts in \eqref{eq:perforated-compact}, define
\begin{equation}\label{eq:evaluation-distance}
 d_N=\sup\bigl\{|h(q)-h(p)|:
                    h\in\calR(K_N),\ \norm h_{K_N}\le1\bigr\}.
\end{equation}
Then $d_N\to0$.
\end{lem}

\begin{proof}
Suppose, to the contrary, that $d_N\nrightarrow0$. Then there exist
$\varepsilon>0$, a strictly increasing sequence $N_m\to\infty$,
and functions $h_m\in\calR(K_{N_m})$ such that
\[
 \norm{h_m}_{K_{N_m}}\le1,
 \qquad
 |h_m(q)-h_m(p)|\ge\varepsilon .
\]
By the final assertion of Lemma~\ref{lem:outer-boundary}, every compact
subset of $\C\setminus Z$ is eventually contained, together with a
neighborhood, in $\interior K_{N_m}$. Hence the functions $h_m$ are
eventually holomorphic and uniformly bounded on a neighborhood of each
compact subset of $\C\setminus Z$.

By Montel's theorem \cite[p.~153, Theorem~2.9]{Con78}, there is a
subsequence, still indexed by $m$, converging locally uniformly on
$\C\setminus Z$ to a holomorphic function $h$ with $|h|\le1$. By Lemma~\ref{lem:skeleton}, every bounded
holomorphic function on $\C\setminus Z$ is constant. Thus $h(p)=h(q)$.
On the other hand, local uniform convergence at $p$ and $q$ gives
\[
 |h(q)-h(p)|
   =\lim_{m\to\infty}|h_m(q)-h_m(p)|
   \ge\varepsilon,
\]
a contradiction.
\end{proof}

Fix $0<\rho<1/2$. By Lemma~\ref{lem:close-evaluations}, we may choose
$N$ sufficiently large so that
\begin{equation}\label{eq:alpha-choice}
 \alpha=2^{-\rho},
 \qquad
 d_N<\frac{1-\alpha}{1+\alpha},
 \qquad
 \alpha<\beta<\frac{1-d_N}{1+d_N}.
\end{equation}
Fix such an $N$ and a choice of $\beta$ for the remainder of the
construction. We use the abbreviations
\[
 K=K_N,\qquad \Omega_{\mathrm{out}}=\Omega_N,\qquad
 H_k=H_{N,k},\qquad L=L_N,\qquad d=d_N.
\]
Write $H_k=B(b_k,t_k)$, where $b_k\in\C$ and $t_k>0$, for
$1\leq k\leq L$, and put
\[
 \Gamma=\partial\Omega_{\mathrm{out}},\qquad
 \Omega=\interior K.
\]
The symbols $\Omega_{\mathrm{out}}$ and $\Omega$ retain these meanings
throughout the paper; the fixed evaluation points remain $p=-1$ and
$q=-1/2$. Then $\partial K$ consists of the outer boundary $\Gamma$,
the boundary circles of the holes contained in $\Omega_{\mathrm{out}}$,
and the point $0$.

\begin{lem}\label{lem:cayley-domination}
If $h\in\calR(K)$ and $\Rea h\ge0$ on $\partial K$, then
\begin{equation}\label{eq:positive-difference}
 \Rea h(q)\ge\beta\Rea h(p).
\end{equation}
There is a positive measure $\tau$ supported on $\partial K$ such that
\[
 \int_{\partial K} h\dd\tau=h(q)-\beta h(p)
 \quad(h\in\calR(K)),\qquad \tau(\partial K)=1-\beta.
\]
\end{lem}

\begin{proof}
Since $\Rea h\ge0$ on $\partial K$, the maximum principle
\cite[p.~7, Corollary~1.9]{ABR01} gives
$\Rea h\ge0$ throughout $K$.
Suppose first that $\Rea h(p)>0$ and set
\[
 \phi(w)=\frac{h(w)-h(p)}
              {h(w)+\overline{h(p)}}.
\]
Because $h(K)$ lies in the closed right half-plane, the denominator
does not vanish on $K$. Hence $\phi\in\calR(K)$ and
$\norm{\phi}_K\le1$. Moreover, $\phi(p)=0$. Thus, by
\eqref{eq:evaluation-distance},
\[
 |\phi(q)|\le d.
\]
Solving this inequality for the real part of $h(q)$ gives the standard
Cayley-transform estimate
\[
 \frac{\Rea h(q)}{\Rea h(p)}
 \ge \frac{1-d}{1+d}
 >\beta,
\]
which proves \eqref{eq:positive-difference} in this case. If
$\Rea h(p)=0$, the same inequality follows immediately from
$\Rea h(q)\ge0$.

Now consider the real linear space
\[
 \mathcal A
  =\{\Rea h|_{\partial K}:h\in\calR(K)\}
  \subset C(\partial K,\R),
\]
and define
\[
\Lambda_\beta(\Rea h|_{\partial K})
  =\Rea h(q)-\beta\Rea h(p).
\]
This is well defined: if $\Rea h=0$ on $\partial K$, then the maximum
principle,
applied to both $\Rea h$ and $-\Rea h$, gives
$\Rea h=0$ on $K$, and in particular at $p$ and $q$.
By \eqref{eq:positive-difference}, the functional $\Lambda_\beta$ is positive, and
\[
 \Lambda_\beta(1)=1-\beta.
\]
Lemma~\ref{lem:positive-extension} therefore yields a positive
measure
$\tau$ on $\partial K$ such that
\[
\int_{\partial K} \Rea h\,\dd\tau
   =\Rea\bigl(h(q)-\beta h(p)\bigr),
 \qquad h\in\calR(K).
\]
Applying the same identity to $-ih$ gives the corresponding equality
for the imaginary parts. Hence
\[
\int_{\partial K} h\,\dd\tau=h(q)-\beta h(p),
\]
as required.
\end{proof}

A positive representing measure need not assign positive mass to any
particular boundary circle. The next lemma shows how to construct one
that carries a definite amount of normalized arclength measure on each
basic hole, without imposing any global regularity assumption on the
infinitely perforated boundary.

\begin{lem}\label{lem:boundary-reservoir}
There is a positive probability measure $\omega$ on $\partial K$
representing evaluation at $p$ on $\calR(K)$, and constants
$\varepsilon_k>0$, such that
\begin{equation}\label{eq:boundary-reservoir}
 \int_{\partial K} h\dd\omega=h(p),\qquad
 \omega\ge\sum_{k=1}^L\varepsilon_k\lambda_k,
\end{equation}
where $\lambda_k$ is normalized arclength measure on $\partial H_k$:
\[
 d\lambda_k=\frac{1}{2\pi t_k}\,ds\big|_{\partial H_k},
 \qquad 1\leq k\leq L.
\]
\end{lem}

\begin{proof}
Since the boundary circle
$\partial H_k$ is isolated from the remaining holes, there exists
$\delta_k>0$ such that
\[
 t_k<|w-b_k|<t_k+\delta_k
\]
is contained in $\Omega$. Let $h\in\calR(K)$ and set
$u=\Rea h$, assuming that $u\ge0$ on $\partial K$. By the maximum principle,
$u\ge0$ on $K$.

The circular mean of $h$ on the above annulus is independent of the
radius, by the Laurent expansion of $h$ there. Hence, by continuity up
to the inner boundary,
\[
 \int u\,\dd\lambda_k
 =\frac1{2\pi}\int_0^{2\pi}
   u\bigl(b_k+(t_k+\delta_k/2)e^{i\theta}\bigr)\,\dd\theta.
\]
Harnack's inequality \cite[p.~48, Theorem~3.6]{ABR01}, applied in the
connected domain $\Omega$ to the point $p$ and the compact circle
\[
 |w-b_k|=t_k+\delta_k/2,
\]
gives a constant $C_k>0$, independent of $h$, such that
\begin{equation}\label{eq:harnack-circle-estimate}
 0\le \int u\,\dd\lambda_k\le C_k\,u(p).
\end{equation} If $u(p)=0$, then $u\equiv0$ on
$\Omega$ by the minimum principle, so the same inequality remains valid.

Choose $\varepsilon_k>0$ so that
\[
 \sum_{k=1}^L \varepsilon_k C_k\le \frac12.
\]
On the real linear space
$
 \mathcal A=\{\Rea h|_{\partial K}:h\in\calR(K)\},
$
define
$
  \Lambda_{\mathrm{res}}(u|_{\partial K})
   =u(p)-\sum_{k=1}^L\varepsilon_k\int u\,\dd\lambda_k.
$
This functional is well defined. Indeed, if
\(u=\Rea h=0\) on \(\partial K\), then the maximum principle applied
to \(u\) and \(-u\) gives \(u\equiv0\) on \(K\). Hence \(u(p)=0\) and
\(\int u\,\dd\lambda_k=0\) for every \(k\).
By \eqref{eq:harnack-circle-estimate} and the choice of the
$\varepsilon_k$ above,
$$
 \Lambda_{\mathrm{res}}(u|_{\partial K})
 \ge
 \left(1-\sum_{k=1}^L\varepsilon_k C_k\right)u(p)
 \ge \frac12 u(p)\ge0.
$$
Thus $\Lambda_{\mathrm{res}}$ is positive.
 By
Lemma~\ref{lem:positive-extension}, there is a positive measure
$\omega_0$ on $\partial K$ representing $\Lambda_{\mathrm{res}}$. Setting
\[
 \omega=\omega_0+\sum_{k=1}^L\varepsilon_k\lambda_k,
\]
we obtain, for every $h\in\calR(K)$,
\[
 \int_{\partial K} \Rea h\,\dd\omega=\Rea h(p).
\]
Applying the same identity to $-ih$ gives the corresponding equality
for the imaginary parts. Hence
\[
 \int_{\partial K} h\,\dd\omega=h(p).
\]
Moreover,
\[
 \omega\ge\sum_{k=1}^L\varepsilon_k\lambda_k.
\]
Finally, testing with the constant function $1$ gives
$\omega(\partial K)=1$.
\end{proof}

Combining the two positive measures yields the difference of point
evaluations needed for the telescoping dilation argument.

\begin{prop}\label{prop:difference-measure}
There is a positive measure $\sigma$ supported on $\partial K$ such that
\begin{equation}\label{eq:sigma-identity}
 \int_{\partial K} h(w)\dd\sigma(w)=h(q)-\alpha h(p)
 \quad(h\in\calR(K)),\qquad \sigma(\partial K)=1-\alpha.
\end{equation}
For every basic hole $H_k$ there is a constant $\kappa_k>0$ with
\begin{equation}\label{eq:basic-circle-mass}
 \sigma|_{\partial H_k}\ge \kappa_k\,ds.
\end{equation}
\end{prop}

\begin{proof}
Set
$
 \sigma=\tau+(\beta-\alpha)\omega.
$
By Lemmas~\ref{lem:cayley-domination} and
\ref{lem:boundary-reservoir}, we have, for every
$h\in\calR(K)$,
\[
 \int_{\partial K} h\,\dd\sigma
   =h(q)-\beta h(p)+(\beta-\alpha)h(p)
   =h(q)-\alpha h(p).
\]
Also,
$
 \sigma(\partial K)=(1-\beta)+(\beta-\alpha)=1-\alpha.
$
Finally, since $\beta>\alpha$ and, by \eqref{eq:boundary-reservoir},
$
 \omega\ge \sum_{k=1}^L \varepsilon_k\lambda_k,
$
the restriction of $\sigma$ to each basic circle satisfies
\[
 \sigma|_{\partial H_k}
 \ge (\beta-\alpha)\varepsilon_k\lambda_k.
\]
Thus \eqref{eq:basic-circle-mass} holds with
\[
 \kappa_k=\frac{(\beta-\alpha)\varepsilon_k}{2\pi t_k}>0.
\]
\end{proof}

\section{A self-similar logarithmic potential}\label{sec:potential}

We sum the inverted boundary measure over all integer scales. The
resulting measure is invariant under the map $z\mapsto2z$, and its Cauchy
transform telescopes on the prescribed disks.

Let $I:\C\setminus\{0\}\to\C\setminus\{0\}$ be the inversion
$I(w)=1/w$. Since $H_k=B(b_k,t_k)$ and $t_k<|b_k|$, its image is
\[
 D_k=I(H_k)=B(a_k,s_k),\qquad
 a_k=\frac{\overline{b_k}}{|b_k|^2-t_k^2},\qquad
 s_k=\frac{t_k}{|b_k|^2-t_k^2}.
\]
For $j\in\Z$ and $1\leq k\leq L$, define
\[
 a_{j,k}=2^ja_k,\qquad R_{j,k}=2^js_k,\qquad
 D_{j,k}=2^jD_k=B(a_{j,k},R_{j,k}).
\]
In particular, $D_{0,k}=D_k$. These disks have pairwise disjoint
closures and lie in a fixed right-half-plane sector.

\begin{prop}\label{prop:power-potential}
Let $\sigma$ be the measure from
Proposition~\ref{prop:difference-measure}, and put
\[
 \sigma_0=\sigma-\sigma(\{0\})\delta_0,\qquad \nu=I_*\sigma_0.
\]
The push-forward is taken on $\C\setminus\{0\}$, where $\sigma_0$
carries all its mass.
With $\alpha$ as in \eqref{eq:alpha-choice}, define
\begin{equation}\label{eq:bilateral-measure}
 \mu=\sum_{n\in\Z}\alpha^{-n-1}(2^n)_*\nu,
 \qquad \alpha=2^{-\rho},
\end{equation}
where $(2^n)_*$ denotes push-forward by $\zeta\mapsto2^n\zeta$.
Then $\mu$ is a nonzero positive Radon measure on $\C$ and
\begin{equation}\label{eq:mu-growth}
 \mu(\{0\})=0,\qquad
 n_\mu(r):=\mu(\{|z|\le r\})\le Cr^\rho\quad(r>0).
\end{equation}
For every Borel set $\mathcal B\subset\C$,
\begin{equation}\label{eq:measure-similarity}
 \mu(2\mathcal B)=2^\rho\mu(\mathcal B).
\end{equation}
The normalized logarithmic potential
\begin{equation}\label{eq:U-definition}
 U(z)=\int_\C\log\left|1-\frac z\zeta\right|\dd\mu(\zeta)
\end{equation}
is nonconstant and subharmonic, has Riesz measure $\mu$, and satisfies
\begin{equation}\label{eq:U-similarity}
 U(0)=0,\qquad U(2z)=2^\rho U(z)\quad(z\in\C).
\end{equation}
Moreover, there is a constant $C_U>0$ such that
\begin{equation}\label{eq:U-upper}
 U(z)\le C_U|z|^\rho,
\end{equation}
and there are real constants $c_k$, the values of $U$ on the basic
disks $D_k$, such that
\begin{equation}\label{eq:flat-values}
 U\equiv c_{j,k}:=2^{j\rho}c_k\quad\hbox{on }D_{j,k}
 \qquad(j\in\Z,\ 1\le k\le L).
\end{equation}
\end{prop}

\begin{proof}
Write $R_K=\max_{w\in K}|w|$. The measure $\nu$ is finite and
supported in $\{|\zeta|\ge R_K^{-1}\}$. Consequently, for every $r>0$,
\[
 n_\mu(r)\le \nu(\C)
    \sum_{n\le\log_2(R_Kr)}\alpha^{-n-1}\le Cr^\rho.
\]
Here convergence at the negative end uses $\rho>0$. Thus $\mu$ is
locally finite, and $\mu(\{0\})=0$. It is nonzero by
\eqref{eq:basic-circle-mass}. Reindexing the positive series
\eqref{eq:bilateral-measure} gives
$(2)_*\mu=\alpha\mu$, which is equivalent to
\eqref{eq:measure-similarity}.

The growth estimate implies
\[
 \int_{0<|\zeta|<1}\log\frac1{|\zeta|}\,\dd\mu(\zeta)<\infty,
 \qquad
 \int_{|\zeta|>1}\frac{\dd\mu(\zeta)}{|\zeta|}<\infty.
\]
The first assertion follows by integrating $n_\mu(t)/t$ over $(0,1)$;
the second follows by integration by parts and $\rho<1$. On a bounded
part of the measure, write the kernel as
$\log|z-\zeta|-\log|\zeta|$. Its potential is subharmonic, with a
finite normalizing constant. The tail is locally uniformly convergent
and harmonic. This defines \eqref{eq:U-definition}, gives $U(0)=0$,
and shows that $(2\pi)^{-1}\Delta U=\mu$. In particular, $U$ is
nonconstant. Changing variables in \eqref{eq:U-definition}, using
\eqref{eq:measure-similarity}, gives \eqref{eq:U-similarity} exactly.
Also, with $r=|z|>0$, integration by parts gives
\[
 U(z)\le\int_\C\log\left(1+\frac r{|\zeta|}\right)\dd\mu(\zeta)
 =r\int_0^\infty\frac{n_\mu(t)}{t(t+r)}\,\dd t
 \le C_Ur^\rho.
\]
Both boundary terms vanish by $0<\rho<1$. This proves
\eqref{eq:U-upper}.

It remains to prove flatness. Fix $z\in D_{j,k}$ and $n\in\Z$.
The function
\[
 h_{n,z}(w)=\frac{w}{1-2^{-n}zw}
\]
vanishes at $0$ and has its only pole at
\[
 \frac{2^n}{z}\in 2^{n-j}H_k\subset\C\setminus K.
\]
Thus $h_{n,z}\in\calR(K)$, even when $n<0$, and
\eqref{eq:sigma-identity} applies with $\sigma_0$ in place of $\sigma$.
Using $p=-1$ and $q=-1/2$, we obtain
\[
 \int_\C\frac{\dd((2^n)_*\nu)(\zeta)}{z-\zeta}
 =-2^{-n}\int h_{n,z}\,\dd\sigma_0
 =\frac1{z+2^{n+1}}-\frac{\alpha}{z+2^n}.
\]
For nonnegative integers $M_-$ and $M_+$, summing gives
\begin{equation}\label{eq:bilateral-telescoping}
 \sum_{n=-M_-}^{M_+}\alpha^{-n-1}
 \int_\C\frac{\dd((2^n)_*\nu)(\zeta)}{z-\zeta}
 =\frac{\alpha^{-M_+-1}}{z+2^{M_++1}}
  -\frac{\alpha^{M_-}}{z+2^{-M_-}}.
\end{equation}
The first term tends to zero as $M_+\to\infty$ because $\rho<1$,
and the second tends to zero as $M_-\to\infty$ because $\rho>0$.
These truncation indices are independent of the fixed geometric stage $N$.
Each summand measure gives zero mass to $D_{j,k}$: its inverse image
there under $w\mapsto2^n/w$ is the hole $2^{n-j}H_k$, which is
disjoint from $K$. Since \(D_{j,k}\) is open and \(\mu(D_{j,k})=0\), we have
\(D_{j,k}\cap\supp\mu=\varnothing\). Hence every compact subset of
\(D_{j,k}\) has positive distance from \(\supp\mu\). Local finiteness of $\mu$,
together with the estimate
$|z-\zeta|^{-1}\leq 2|\zeta|^{-1}$ for sufficiently large
$|\zeta|$, uniformly for $z$ in such a compact subset, and
\eqref{eq:mu-growth} with $\rho<1$, shows that the Cauchy integral
against $\mu$ converges absolutely and locally uniformly on
$D_{j,k}$. Passing to the limit in
\eqref{eq:bilateral-telescoping} yields
\[
 2\p U(z)=\int_\C\frac{\dd\mu(\zeta)}{z-\zeta}=0.
\]
Thus $U$ is constant on each $D_{j,k}$. The relation between these
constants in \eqref{eq:flat-values} follows from
\eqref{eq:U-similarity}.
\end{proof}

\begin{lem}\label{lem:uniform-geometry}
There is $\kappa>0$ such that, for every $j\in\Z$ and $1\le k\le L$,
\begin{equation}\label{eq:circle-density}
 \mu|_{\partial D_{j,k}}\ge\kappa\,2^{j(\rho-1)}\,ds.
\end{equation}
There is $q_0\in(0,1)$ such that
\begin{equation}\label{eq:radial-disk-cover}
 \calU=\bigcup_{\substack{j\in\Z\\1\leq k\leq L}}D_{j,k}(q_0)
\end{equation}
meets every circle centered at zero and satisfies $2\calU=\calU$.
For every fixed integer $J$, the subfamily with $j\ge J$ meets every
sufficiently large circle.
\end{lem}

\begin{proof}
By \eqref{eq:basic-circle-mass} and smoothness of inversion on each
basic circle, $\nu|_{\partial D_k}$ dominates a positive multiple
of arclength. The term $n=0$ in \eqref{eq:bilateral-measure} then
gives the same assertion for $\mu|_{\partial D_k}$. Equation
\eqref{eq:measure-similarity}, together with the scaling of arclength,
gives \eqref{eq:circle-density}; take the minimum over the finitely
many $k$.

The compact set $I(E)$ is contained in the finite open union
$\bigcup_{k=1}^LD_k$. Thus some $q_0\in(0,1)$ satisfies
\[
 I(E)\subset\bigcup_{k=1}^LD_k(q_0).
\]
By \eqref{eq:radial-model}, its radial projection is $[e^{-\ell},1]$,
whose dyadic dilates cover
$(0,\infty)$ because $\ell>\log2$. This proves the asserted radial
coverage in \eqref{eq:radial-disk-cover}. Dyadic invariance is immediate
from its definition. Finally, the disks with $j<J$ form a bounded
family and cannot meet sufficiently large circles.
\end{proof}

\section{From the subharmonic function to an entire function}
\label{sec:from-subharmonic}

Throughout this section, $U$ and $\mu$ are the potential and measure
from Proposition~\ref{prop:power-potential}. Fix $\kappa>0$ and
$q_0\in(0,1)$ as in Lemma~\ref{lem:uniform-geometry}, and let
$\calU$ be the corresponding open set in \eqref{eq:radial-disk-cover}.
Then $U(0)=0$, $U(2z)=2^\rho U(z)$, and $U$ is locally constant on
$\calU$, which intersects every circle $|z|=r$, $r>0$.
For use throughout the rest of the section, put
\[
 \gamma_k=\pi\kappa s_k,\qquad
 \gamma_{j,k}=2^{j\rho}\gamma_k
 \qquad(j\in\Z,\ 1\leq k\leq L).
\]

\begin{lem}\label{lem:green-modification}
Let $U$, $\mu$, and $D_{j,k}$ be as in
Proposition~\ref{prop:power-potential} and
Lemma~\ref{lem:uniform-geometry}.
There exists a subharmonic function $U_1$ in $\C$ such that
\[
U_1(0)=0,\qquad
U_1(2z)=2^\rho U_1(z),\qquad
U_1\leq U,
\]
and
\[
U_1<U\quad\hbox{on every }D_{j,k}.
\]
\end{lem}

\begin{proof}
Using the centers $a_{j,k}$ and radii $R_{j,k}$ fixed in
Section~\ref{sec:potential}, define
\[
G_{j,k}(z)=
\begin{cases}
\log\bigl(|z-a_{j,k}|/R_{j,k}\bigr),
   & |z-a_{j,k}|<R_{j,k},\\
0, & |z-a_{j,k}|\geq R_{j,k}.
\end{cases}
\]
In the sense of distributions,
\[
\frac{1}{2\pi}\Delta G_{j,k}
=
\delta_{a_{j,k}}
-
\frac{1}{2\pi R_{j,k}}\,
ds\big|_{\partial D_{j,k}}.
\]

With the coefficients $\gamma_{j,k}$ fixed above, define, for $z\neq0$,
\[
U_1(z)=U(z)+
\sum_{j\in\Z}\sum_{k=1}^L
\gamma_{j,k}G_{j,k}(z).
\]
The sum is locally finite on $\C\setminus\{0\}$.
Moreover,
\[
\frac{\gamma_{j,k}}{2\pi R_{j,k}}
=
\frac{\kappa}{2}\,2^{j(\rho-1)}.
\]
Consequently, \eqref{eq:circle-density} implies that
\[
\mu+
\sum_{j\in\Z}\sum_{k=1}^L\gamma_{j,k}\delta_{a_{j,k}}
-
\sum_{j\in\Z}\sum_{k=1}^L
\frac{\gamma_{j,k}}{2\pi R_{j,k}}\,
ds\big|_{\partial D_{j,k}}
\]
is a positive measure on $\C\setminus\{0\}$.
Thus $U_1$ is subharmonic there.

Since $G_{j,k}\leq0$, we have
\[
U_1(z)\leq U(z)\leq C_U|z|^\rho.
\]
Hence $U_1$ has a subharmonic extension across $0$.
Choose a point $z_*$ in the left half-plane at which $U$ is finite.
All the disks lie in the right half-plane, and therefore
\[
U_1(2^{-n}z_*)
=
U(2^{-n}z_*)
=
2^{-n\rho}U(z_*)\longrightarrow0.
\]
It follows that the extended value is $U_1(0)=0$.

Finally,
\[
G_{j+1,k}(2z)=G_{j,k}(z),
\qquad
\gamma_{j+1,k}=2^\rho\gamma_{j,k},
\]
so $U_1(2z)=2^\rho U_1(z)$.
The strict inequality on each $D_{j,k}$ follows from
$G_{j,k}<0$ there and $\gamma_{j,k}>0$.
\end{proof}

\begin{prop}
\label{prop:subharmonic-to-entire}
For the potential $U$ and the open set $\calU$ constructed above,
there exists an entire function $F$ such that
\begin{equation}
\label{1}
c_Tr^\rho\leq T(r,F)\leq C_Tr^\rho
\end{equation}
for some constants $c_T,C_T>0$ and all sufficiently large $r$, and
\begin{equation}
\label{2}
\delta(0,F'/F)>0.
\end{equation}
\end{prop}

\begin{proof}
Let $U_1$ be the function constructed in
Lemma~\ref{lem:green-modification}.
We construct $F$ directly by a weighted $\bar\partial$ argument.

\medskip
\noindent\textbf{A weighted $\bar\partial$ estimate.}
We first give the one-variable weighted estimate that we use; it is a
special case of the $L^2$ method of H\"ormander \cite{Hor65}.
If $v\not\equiv-\infty$ is subharmonic in $\C$ and
\[
 \mathcal E_b:=\int_\C |b|^2e^{-2v}\dA<\infty,
\]
then there is a distributional solution $h$ of $\db h=b$ satisfying
\begin{equation}\label{eq:weighted-dbar}
 \int_\C \frac{|h|^2e^{-2v}}{(1+|z|^2)^2}\dA\leq\frac{\mathcal E_b}{2}.
\end{equation}
Here $b$ is a measurable scalar function, identified with the
coefficient of the $(0,1)$-form $b\,d\bar z$.

For completeness we sketch a proof. Suppose first that $v$ is smooth, and set
\[
 \Phi=2v+2\log(1+|z|^2),\qquad
 \Phi_{z\bar z}\geq\frac{2}{(1+|z|^2)^2}.
\]
Here $\Phi_z=\p\Phi$ and $\Phi_{z\bar z}=\p\db\Phi$.
Write $\langle\cdot,\cdot\rangle_\Phi$ and $\|\cdot\|_\Phi$ for the
inner product and norm in $L^2(e^{-\Phi}\dA)$, with
$\langle f,g\rangle_\Phi=\int_\C \overline f\,g\,e^{-\Phi}\dA$.
The formal adjoint of $\db$ is
$\db_\Phi^*\psi=-\p\psi+\Phi_z\psi$.
Integration by parts gives, for $\psi\in C_c^\infty(\C)$,
\[
 \|\db_\Phi^*\psi\|_\Phi^2
 =\|\db\psi\|_\Phi^2
   +\int_\C\Phi_{z\bar z}|\psi|^2e^{-\Phi}\dA.
\]
Consequently,
\begin{equation}\label{eq:smooth-dbar-estimate}
 \begin{aligned}
 |\langle b,\psi\rangle_\Phi|^2
 &\leq
 \left(\int_\C\frac{|b|^2}{\Phi_{z\bar z}}e^{-\Phi}\dA\right)
 \left(\int_\C\Phi_{z\bar z}|\psi|^2e^{-\Phi}\dA\right)\\
 &\leq\frac{\mathcal E_b}{2}\,\|\db_\Phi^*\psi\|_\Phi^2.
 \end{aligned}
\end{equation}
The Hahn--Banach theorem and the Riesz representation theorem, applied
to the functional on the range of $\db_\Phi^*$ defined by this pairing,
give $h$ with $\db h=b$ and \eqref{eq:weighted-dbar}.

For a general subharmonic \(v\), choose a decreasing family of smooth
subharmonic regularizations \(v_\epsilon\downarrow v\) as
\(\epsilon\downarrow0\).
Since $v_\epsilon\geq v$, the bound in
\eqref{eq:smooth-dbar-estimate} remains at most
$\mathcal E_b/2$, independently of $\epsilon$.
The resulting solutions $h_\epsilon$ are bounded in $L^2$ on every
compact set, since $v_\epsilon$ is locally uniformly bounded above.
A diagonal weakly convergent subsequence has a limit $h$ satisfying
$\db h=b$.
For each fixed $\delta>0$ and $\epsilon<\delta$, the same bound holds
with the weight $e^{-2v_\delta}(1+|z|^2)^{-2}$.
Weak lower semicontinuity on compact disks, followed by exhaustion of
$\C$ and then monotone convergence as $\delta\downarrow0$, proves
\eqref{eq:weighted-dbar} for $v$.

\medskip
\noindent\textbf{Construction of the entire function.}
We use the function $U_1$ constructed in
Lemma~\ref{lem:green-modification}, the radii $R_{j,k}=2^js_k$
from Section~\ref{sec:potential}, and the coefficients
$\gamma_{j,k}=2^{j\rho}\gamma_k>0$ fixed at the start of this section.
The disks have disjoint closures, so
\eqref{eq:flat-values} gives the exact formula
\begin{equation}\label{eq:local-green-weight}
 U_1(z)=c_{j,k}+\gamma_{j,k}
          \log\frac{|z-a_{j,k}|}{R_{j,k}}
 \qquad (z\in D_{j,k}).
\end{equation}
Choose
\[
 q_0<q_1<q_2<q_3<q_4<1,
\]
where $q_0$ is supplied by Lemma~\ref{lem:uniform-geometry}.
Let $\chi\in C_c^\infty(B(0,q_4))$ be radial, with
$0\leq\chi\leq1$ and $\chi=1$ on $\overline{B(0,q_3)}$.
For $j\geq0$ and $1\leq k\leq L$, put
\[
 A_{j,k}=\exp\bigl(c_{j,k}+\gamma_{j,k}\log q_2\bigr),
 \qquad
 \chi_{j,k}(z)=\chi\left(\frac{z-a_{j,k}}{R_{j,k}}\right),
\]
and define the smooth function
\[
 F_0(z)=\sum_{j=0}^{\infty}\sum_{k=1}^L
                A_{j,k}\chi_{j,k}(z).
\]
This sum is locally finite, its supports are disjoint, and $F_0$ vanishes
on a neighborhood of the origin. Set
\[
 b=\db F_0,\qquad v(z)=U_1(z)+\log|z|.
\]
The function $v$ is subharmonic on $\C$ and is not identically $-\infty$.
On the support of $\db\chi_{j,k}$ we have
\[
 U_1\geq c_{j,k}+\gamma_{j,k}\log q_3,
 \qquad |\db\chi_{j,k}|\leq C R_{j,k}^{-1}.
\]
Moreover, $|z|$ is bounded above and below by positive constant
multiples of $2^j$ on $D_{j,k}$, uniformly in $k$.
It follows that
\begin{equation}\label{eq:source-integrability}
 \begin{aligned}
 \int_\C|b|^2e^{-2v}\dA
 &\leq C\sum_{j=0}^{\infty}\sum_{k=1}^L
  2^{-2j}\exp\left(-2\gamma_{j,k}\log\frac{q_3}{q_2}\right)
 <\infty.
 \end{aligned}
\end{equation}
In view of \eqref{eq:source-integrability}, apply
\eqref{eq:weighted-dbar} and choose $h$ with $\db h=b$ and
\begin{equation}\label{eq:correction-norm}
 \mathcal E_h:=\int_\C
 \frac{|h(z)|^2e^{-2U_1(z)}}{|z|^2(1+|z|^2)^2}\dA(z)<\infty.
\end{equation}
Then $F=F_0-h$ has an entire representative, which we use from now on.
In particular, $h$ is holomorphic wherever $b=0$, including a
neighborhood of $0$ and each $D_{j,k}(q_3)$ with $j\geq0$.
Since $U_1(z)\leq C_U|z|^\rho$,
\eqref{eq:correction-norm} implies
\[
 \int_{|z|<r_0}|h(z)|^2|z|^{-2}\dA(z)<\infty
\]
for some $r_0>0$. Holomorphy at $0$ therefore forces $h(0)=0$,
and hence $F(0)=0$.

\medskip
\noindent\textbf{Exponential flatness on the smaller disks.}
On $D_{j,k}(q_1)$, formula \eqref{eq:local-green-weight} gives
$U_1\leq c_{j,k}+\gamma_{j,k}\log q_1$.
Since $|z|^2(1+|z|^2)^2\leq C2^{6j}$ there, for $j\geq0$,
\eqref{eq:correction-norm} yields
\[
 \int_{D_{j,k}(q_1)}|h|^2\dA
 \leq C\mathcal E_h\,2^{6j}
       \exp\bigl(2c_{j,k}+2\gamma_{j,k}\log q_1\bigr).
\]
The mean-value inequality and Cauchy's estimate on the concentric
disks with radii $q_0R_{j,k}$ and $q_1R_{j,k}$ now give
\begin{equation}\label{eq:local-correction}
 \begin{aligned}
 \sup_{D_{j,k}(q_0)}\bigl(|h|+R_{j,k}|h'|\bigr)
 &\leq\frac{C}{R_{j,k}}
       \left(\int_{D_{j,k}(q_1)}|h|^2\dA\right)^{1/2}\\
 &\leq C2^{2j}
       \exp\bigl(c_{j,k}+\gamma_{j,k}\log q_1\bigr).
 \end{aligned}
\end{equation}
All constants here are independent of $j$ and $k$.
Set
\[
 \varepsilon=\frac12\min_{1\leq k\leq L}
                   \gamma_k\log\frac{q_2}{q_1}>0.
\]
After division by $A_{j,k}$, the last bound in
\eqref{eq:local-correction} is at most
$C2^{2j}\exp(-2\varepsilon2^{j\rho})$.
As $\rho>0$, there is an integer $J\geq0$ such that, for all $j\geq J$,
\[
 \sup_{D_{j,k}(q_0)}\bigl(|h|+R_{j,k}|h'|\bigr)
 \leq A_{j,k}\exp(-\varepsilon2^{j\rho}),
 \qquad 1\leq k\leq L.
\]
Increase $J$ so that also $R_{j,k}\geq2$ and
$\exp(-\varepsilon2^{j\rho})\leq1/2$ for all these indices.
On $D_{j,k}(q_0)$ we have $F_0=A_{j,k}$, and consequently
\[
 |F|\geq\frac{A_{j,k}}2,
 \qquad
 |F'|\leq\frac{A_{j,k}}{R_{j,k}}
                    \exp(-\varepsilon2^{j\rho}).
\]
In particular, $F$ is not identically zero and, since $F(0)=0$, it is
nonconstant. We obtain
\begin{equation}\label{eq:flat-logarithmic-ratio}
 \log\left|\frac{F(z)}{F'(z)}\right|
 \geq\varepsilon2^{j\rho}
 \qquad
 (z\in D_{j,k}(q_0),\ j\geq J).
\end{equation}
At a zero of $F'$ the left-hand side is interpreted as $+\infty$.

\medskip
\noindent\textbf{Global growth and a uniform angular bound.}
By \eqref{eq:U-upper}, \eqref{eq:flat-values}, and the disjointness of
the cutoff supports,
\[
 |F_0(z)|\leq\exp(C_U|z|^\rho).
\]
Also, \eqref{eq:correction-norm} and $U_1\leq C_U|z|^\rho$ imply,
for $|z|\geq2$,
\[
 \int_{B(z,1)}|h|^2\dA
 \leq C\mathcal E_h(1+|z|)^6\exp\bigl(2C_U(|z|+1)^\rho\bigr).
\]
Applying the submean inequality to the entire function $F=F_0-h$
and using $|F_0-h|^2\leq2|F_0|^2+2|h|^2$, we conclude that
\begin{equation}\label{eq:direct-upper-growth}
 \log M(r,F)\leq A r^\rho
\end{equation}
for some $A>0$ and every sufficiently large $r$.

For $r>0$, define the angular measure
\[
 \eta(r)=\int_{-\pi}^{\pi}
             \mathbf{1}_{\calU}(re^{i\theta})\,d\theta,
\]
where $\mathbf{1}_{\calU}$ is the indicator function of $\calU$.
Every circle meets the open set $\calU$, so $\eta(r)>0$.
Only finitely many of the disks in \eqref{eq:radial-disk-cover} meet
the compact annulus $1\leq|z|\leq2$.
For each such disk, its boundary meets any circle centered at $0$ in
at most two points, since its center is nonzero.
Dominated convergence therefore proves that $\eta$ is continuous on
$[1,2]$. Hence
\begin{equation}\label{eq:uniform-angular-measure}
 \eta_*:=\min_{1\leq r\leq2}\eta(r)>0,
 \qquad \eta(r)\geq\eta_*\quad(r>0),
\end{equation}
where the second assertion follows from $2\calU=\calU$.
Put
\[
 R_*:=\max_{1\leq k\leq L}(|a_k|+q_0s_k).
\]
If $|z|=r$ and $z\in D_{j,k}(q_0)$, then $r\leq2^jR_*$.
For sufficiently large $r$, all disks meeting $|z|=r$ have $j\geq J$.
Combining \eqref{eq:flat-logarithmic-ratio} and
\eqref{eq:uniform-angular-measure} gives
\begin{equation}\label{eq:direct-proximity-lower}
 m\!\left(r,\frac{F}{F'}\right)
 \geq\frac{\varepsilon\eta_*}{2\pi R_*^\rho}r^\rho
 =:Br^\rho
\end{equation}
for every sufficiently large $r$.

\medskip
\noindent\textbf{Two-sided characteristic growth and deficiency.}
For an entire function, the Poisson estimate for
$\log^+|F|$ gives $\log M(s,F)\leq3T(2s,F)$ for large $s$.
Thus Lemma~\ref{lem:characteristic} and
\eqref{eq:direct-proximity-lower} imply
\[
 \begin{aligned}
 Br^\rho
 &\leq m\!\left(r,\frac{F}{F'}\right)
 \leq T\!\left(r,\frac{F'}F\right)+O(1)\\
 &\leq2\log M(2r,F)+O(\log r)
 \leq6T(4r,F)+O(\log r).
 \end{aligned}
\]
Since $\log r=o(r^\rho)$, replacing $r$ by $r/4$ proves the lower
bound in \eqref{1}. The upper bound follows from
$T(r,F)\leq\log^+M(r,F)$ and \eqref{eq:direct-upper-growth}.
In particular, $F$ is transcendental and has order and lower order
$\rho$.
Using \eqref{eq:direct-upper-growth} once more in
Lemma~\ref{lem:characteristic}, we obtain a constant $C>0$ such that
\[
 T\!\left(r,\frac{F'}F\right)\leq Cr^\rho
\]
for all sufficiently large $r$. Consequently,
\[
 \delta\!\left(0,\frac{F'}F\right)
 =\liminf_{r\to\infty}
   \frac{m(r,F/F')}{T(r,F'/F)}
 \geq\frac BC>0,
\]
which proves \eqref{2}.
Finally, \eqref{1} and $T(r,F)\leq\log^+M(r,F)$ give the lower
bound for $\log M(r,F)$ in \eqref{eq:main-growth}; its other two
bounds are \eqref{eq:direct-upper-growth} and
\eqref{eq:direct-proximity-lower}.
This completes the proof of Theorem~\ref{thm:main}.
\end{proof}

\end{document}